\documentclass{article}

\usepackage[utf8]{inputenc}
\usepackage{amsmath}
\usepackage{amsthm}
\usepackage{amssymb}
\usepackage{fontspec}
\usepackage{color}
\usepackage[pdfusetitle,
 bookmarks=true,bookmarksnumbered=false,bookmarksopen=false,
 breaklinks=true,pdfborder={0 0 0},pdfborderstyle={},backref=false,colorlinks=true]
 {hyperref}
\hypersetup{
 colorlinks=true,linkcolor=violet,urlcolor=cyan,citecolor=magenta}

\usepackage{tikz-cd}
\makeatletter
\theoremstyle{plain}
\newtheorem{thm}{\protect\theoremname}[section]
\theoremstyle{plain}

\usepackage{fullpage}
\date{}
\AtBeginDocument{}

\makeatother

\usepackage{polyglossia}
\setdefaultlanguage[variant=british]{english}
\usepackage[style=alphabetic,maxbibnames=12,maxalphanames=6]{biblatex}

\providecommand{\theoremname}{Theorem}

\usepackage{mathtools}
\usepackage{microtype}
\usepackage{enumitem}
\usepackage{aliascnt}
\usepackage{tikz-cd}
\usepackage[capitalise,
nameinlink
]{cleveref}

\setlist[enumerate]{label=(\roman*),leftmargin=2.2em}
\setlist[itemize]{leftmargin=2.0em}
\allowdisplaybreaks

\crefname{thm}{theorem}{theorems}
\Crefname{thm}{Theorem}{Theorems}
\crefname{conj}{conjecture}{conjectures}
\Crefname{conj}{Conjecture}{Conjectures}

\newaliascnt{ex}{thm}
\newtheorem{ex}[ex]{Example}
\aliascntresetthe{ex}

\crefname{example}{example}{examples}
\Crefname{example}{Example}{Examples}

\newaliascnt{proposition}{thm}
\newtheorem{proposition}[proposition]{Proposition}
\aliascntresetthe{proposition}

\newaliascnt{lemma}{thm}
\newtheorem{lemma}[lemma]{Lemma}
\aliascntresetthe{lemma}

\newaliascnt{corollary}{thm}
\newtheorem{corollary}[corollary]{Corollary}
\aliascntresetthe{corollary}

\theoremstyle{remark}
\newaliascnt{remark}{thm}
\newtheorem{remark}[remark]{Remark}
\aliascntresetthe{remark}
\theoremstyle{plain}

\newaliascnt{conjecture}{thm}
\newtheorem{conjecture}[conjecture]{Conjecture}
\aliascntresetthe{conjecture}

\crefname{conjecture}{conjecture}{conjectures}
\Crefname{conjecture}{Conjecture}{Conjectures}

\providecommand{\im}{\operatorname{im}}
\providecommand{\coker}{\operatorname{coker}}
\providecommand{\id}{\operatorname{id}}

\providecommand{\tensor}{\otimes}

\begin{document}
\global\long\def\rg{R\llbracket G\rrbracket}%
\global\long\def\fp#1{\mathrm{FP}_{#1}}%
\global\long\def\Ker{\ker}%
\global\long\def\operatorname#1{\mathrm{#1}}%
\global\long\def\F{\mathbb{F}}%
\global\long\def\Z{\mathbb{Z}}%
\global\long\def\scarytimes{\boxtimes}%
\global\long\def\di{\dim}%
\global\long\def\dimf{\dim_{\mathbb{F}}}%
\global\long\def\d{\mathrm{d}}%
\global\long\def\llb{\llbracket}%
\global\long\def\rrb{\rrbracket}%
\global\long\def\Hom{\mathrm{Hom}}%
\global\long\def\rep#1{\mathcal{S}^{#1}}%
\global\long\def\supG#1{\sup_{\text{simple }#1\llbracket G\rrb\text{-modules }M}}%
\global\long\def\supRG{\supG R}%
\global\long\def\cb#1#2{\tau_{#1}(#2)}%
\global\long\def\cbf#1#2#3{\tau_{#1}(#2;#3)}%
\global\long\def\ct{\otimes}%
\global\long\def\t{\ct_{R}}%
\global\long\def\te{\ct_{\rg}}%
\global\long\def\li{\underset{\longleftarrow}{\lim}}%
\global\long\def\torg#1{\mathrm{Tor}^{\rg}_{#1}}%
\global\long\def\fr#1{\mathcal{F}\left(#1\right)}%
\global\long\def\graph#1{\mathrm{Graph}\left(#1\right)}%
\global\long\def\FP{\operatorname{FP}}%
\global\long\def\F{\operatorname F}%

\title{Virtual Surjection and the $n$-$(n+1)$-$(n+2)$ Theorem for Discrete Groups}
\author{Tal Cohen and Mark Shusterman}
\maketitle
\begin{abstract}
We prove the Virtual Surjection Conjecture for discrete groups, both
for property $F_{n}$ and for property $FP_{n}$: given a product
of groups of type $F_k$ (respectively $FP_k$), a subgroup that virtually surjects
onto $k$-tuples must be $F_k$ (respectively $FP_k$) as well. We prove the homological
$n$-$(n+1)$-$(n+2)$ Conjecture for discrete groups under the assumption
the common quotient is finitely presented, and prove that this assumption
cannot be dropped. We deduce the $n$-$(n+1)$-$(n+2)$ Conjecture
for property $F_{n}$.
\end{abstract}

\section{Introduction}

A group $G$ is of type $\F_{n}$ if it admits a classifying space
$K(G,1)$ with finite $n$-skeleton. A group is $\F_{1}$ if and only
if it is finitely generated, and $\F_{2}$ if and only if it is finitely
presented. A group $G$ is of type $\FP_{n}$ if the trivial $\Z G$-module
$\Z$ admits a projective resolution which is finitely generated in
degrees $0,\dots,n$. Property $\F_{n}$ implies property $\FP_{n}$;
for finitely presented groups, the two properties are equivalent.
It is a classical result of Bestvina and Brady \cite{BestvinaBrady97}
that this fails for non-finitely-presented groups.

Let $2\leqslant k\leqslant m$ be integers, let $G_{1},\ldots,G_{m}$
be groups, and let $P\leqslant G_{1}\times\cdots\times G_{m}$. We
say that $P$ \emph{virtually surjects }on $k$ factors if, for every
$1\leqslant i_{1}<\cdots<i_{k}\leqslant m$, the image of $P$ in
$G_{i_{1}}\times\cdots\times G_{i_{k}}$ has finite index. In \cite{Kuc14},
Kuckuck gave the following form of the Virtual Surjection Conjecture.
See also \cite{Koc10} and \cite{Dis08}.
\begin{conjecture}[Virtual Surjection Conjecture]
\label{conj:VirSurF}Let $2\leqslant k\leqslant m$ be integers, let
$G_{1},\ldots,G_{m}$ be groups of type $\F_{k}$, and let $P\leqslant G_{1}\times\cdots\times G_{m}$
be a subgroup that virtually surjects on $k$ factors. Then $P$ is
of type $\F_{k}$. 
\end{conjecture}

In \cite{KocLim}, Kochloukova and Lima gave the following homological
Virtual Surjection Conjecture.
\begin{conjecture}[Homological Virtual Surjection Conjecture]
\label{conj:VirSurFP}Let $2\leqslant k\leqslant m$ be integers,
let $G_{1},\ldots,G_{m}$ be groups of type $\FP_{k}$, and let $P\leqslant G_{1}\times\cdots\times G_{m}$
be a subgroup that virtually surjects on $k$ factors. Then $P$ is
of type $\FP_{k}$. 
\end{conjecture}

\Cref{conj:VirSurF} was established for $k=2$ in \cite{BHMS}.
Since property $\FP_{n}$ and $\F_{n}$ are equivalent for finitely
presented groups, this means that \Cref{conj:VirSurFP}
implies \Cref{conj:VirSurF}. Later, \Cref{conj:VirSurFP} was proved for $k=2$ in \cite{KocLim}. 
In this paper, we establish
both conjectures.
\begin{thm}
\label{thm:VirtualSurjection}\Cref{conj:VirSurFP} and
\Cref{conj:VirSurF} are true. Namely, given integers $2\leqslant k\leqslant m$,
groups $G_{1},\ldots,G_{m}$ of type $\F_{k}$ (respectively $\FP_{k}$),
and a subgroup $P\leqslant G_{1}\times\cdots\times G_{m}$ that virtually
surjects on $k$ factors, $P$ is of type $\F_{k}$ (respectively
$\FP_{k}$).
\end{thm}

The standard route to virtual surjection is through fibre products.
Given epimorphisms $G_{1}\to Q$ and $G_{2}\to Q$, their fibre product
is 
\[
G_{1}\times_{Q}G_{2}=\{(g_{1},g_{2})\in G_{1}\times G_{2}\mid\phi_{1}(g_{1})=\phi_{2}(g_{2})\}.
\]
Kuckuck \cite{Kuc14} formulated the following $n$-$(n+1)$-$(n+2)$
conjecture, and showed that, if it holds whenever $Q$ is nilpotent,
then the Virtual Surjection Conjecture holds in general \cite[Theorem 3.12]{Kuc14}.
\begin{conjecture}[The $n$-$(n+1)$-$(n+2)$ Conjecture]
\label{conj:nF}Let $n$ be a nonnegative integer, and let 
\[
1\to N_{1}\to G_{1}\to Q\to1,\qquad1\to N_{2}\to G_{2}\to Q\to1,
\]
be short exact sequences of groups, where $N_{1}$ is $\F_{n}$, $G_{1}$
and $G_{2}$ are $\F_{n+1}$ and $Q$ is $\F_{n+2}$. Then the fibre
product $G_{1}\times_{Q}G_{2}$ is $\F_{n+1}$.
\end{conjecture}

The case $n=0$ is known as the $0$-$1$-$2$ Lemma, see \cite[Lemma 3.8]{Kuc12}. The case $n=1$ was proved in \cite{Dis08,BHMS08} (see \cite{BHMS} for an effective version), and is known as the Asymmetric $1$-$2$-$3$ Theorem. The Symmetric $1$-$2$-$3$ Theorem is the same result for $n=1$ and two identical short exact sequences; it was proved it \cite{BBMS00}.

Kochloukova and Lima \cite{KocLim} gave the corresponding homological conjecture, which implies \Cref{conj:nF}.

\begin{conjecture}[The Homological $n$-$(n+1)$-$(n+2)$ Conjecture]
\label{conj:nFP}Let $n$ be a nonnegative integer, and let 
\[
1\to N_{1}\to G_{1}\to Q\to1,\qquad1\to N_{2}\to G_{2}\to Q\to1,
\]
be short exact sequences of groups, where $N_{1}$ is $\FP_{n}$,
$G_{1}$ and $G_{2}$ are $\FP_{n+1}$ and $Q$ is $\FP_{n+2}$. Then
the fibre product $G_{1}\times_{Q}G_{2}$ is $\FP_{n+1}$.
\end{conjecture}

We establish \Cref{conj:nFP} under the additional assumption
$Q$ is finitely presented, and show that it is false without it (for
every $n$). 
\begin{thm}
\label{thm:mainthm}Let $n$ be a nonnegative integer, and let 
\[
1\to N_{1}\to G_{1}\to Q\to1,\qquad1\to N_{2}\to G_{2}\to Q\to1,
\]
be short exact sequences of groups, where $N_{1}$ is $\FP_{n}$,
$G_{1}$ and $G_{2}$ are $\FP_{n+1}$ and $Q$ is $\FP_{n+2}$ and finitely presented. Then
the fibre product $G_{1}\times_{Q}G_{2}$ is $\FP_{n+1}$.
\end{thm}

\begin{ex}\label{ex:counter}
For every nonnegative integer $n$ there exist short exact sequences
\[
1\to\ker p_{n}\to A_{n}\xrightarrow{p_{n}}Q\to1,\qquad1\to\ker\pi\to F\xrightarrow{}Q\to1,
\]
such that $Q$, $F$ and $A_{n}$ are $\mathrm{FP}_{\infty}$ and
$\ker p_{n}$ is $\mathrm{FP}_{n}$, but the fibre product $A_{n}\times_{Q}F$
is not $\mathrm{FP}_{n+1}$.
\end{ex}

The case $n=1$ of \Cref{thm:mainthm} was proved in \cite{KocLim}. The case $n=0$ of \Cref{ex:counter} is the example from \cite[Proposition 6.13]{Kuc12}.

Since property $\F_{n}$ and $\FP_{n}$ are equivalent for finitely
presented groups, \Cref{thm:mainthm}, together with the $1$-$2$-$3$ Theorem from \cite{BHMS}, imply the $n$-$(n+1)$-$(n+2)$ Conjecture (\Cref{conj:nF}).
\begin{corollary}[The $n$-$(n+1)$-$(n+2)$ Theorem]
    \Cref{conj:nF} is true.
\end{corollary}
Since finitely generated virtually nilpotent groups are finitely presented,
\Cref{thm:mainthm} also implies \Cref{conj:nFP}
under the stronger assumption that $Q$ is virtually nilpotent. 
In \cite[Theorem F]{KocLim}, Kochloukova and Lima showed that, if
\Cref{conj:nFP} holds under the assumption that $Q$ is virtually
nilpotent, then the homological Virtual Surjection Conjecture (\Cref{conj:VirSurFP}) holds in general, which in turn implies the ordinary Virtual Surjection
Conjecture (\Cref{conj:VirSurF}). In other words,
\Cref{thm:mainthm} implies \Cref{thm:VirtualSurjection}. The
bulk of this paper is therefore dedicated to establishing
\Cref{thm:mainthm}.

The Virtual Surjection Conjecture and the $n$-$(n+1)$-$(n+2)$ Conjecture have been considered for property $\FP_n$ of profinite groups. They were established for pro-$p$ groups in \cite{propnn1n2} and for general profinite groups in \cite{CoSu26}. These questions have also been considered in other settings, for example for Lie algebras in \cite{Koc19}.

As a consequence of Theorem 1.3, the conditional result \cite[Theorem D]{KJ22} is now unconditional.
\begin{remark}
    Given a commutative ring with unit $R$, a group $G$ is $\FP_{n}$ over $R$ if the trivial $RG$-module $R$ admits a projective resolution which is finitely generated in degrees $0,\dots,n$. Property $\mathrm{FP}_n$ over $\Z$ implies property $\mathrm{FP}_n$ over $R$ for every $R$. The same proofs show that \Cref{thm:VirtualSurjection,thm:mainthm} remain true for property $\mathrm{FP}_n$ over any $R$.
\end{remark}

\paragraph{Organisation of the paper.} In \Cref{sec:prelim}, we recall some basic homological algebra facts. Throughout \Crefrange{wall-resolutions}{free-second-factor}, we specialise to the case where $G_2$ is a free group of finite rank. In \Cref{wall-resolutions}, we construct resolutions for the trivial modules of the groups involved in the fibre product, following \cite{Wall}. In \Cref{truncation}, we truncate the Wall resolution obtained for the fibre product group, and establish the desired finiteness conditions on it. In \Cref{free-second-factor}, we establish \Cref{thm:mainthm} in the case the second factor is free, and then in \Cref{general-fibre} we deduce the general version, following \cite{KocLim}. In \Cref{sec:counter}, we establish \Cref{ex:counter}.

\paragraph{Acknowledgements.}We developed the main arguments in the body of this paper in the course of interactions with GPT. All proofs were worked out in detail, verified, and written by the authors.

Both authors are co-funded by the European Union (ERC, Function Fields, 101161909). Views and opinions expressed are however those of the authors only and do not necessarily reflect those of the European Union or the European Research Council. Neither the European Union nor the granting authority can be held responsible for them. The second author is The Dr. A. Edward Friedmann Career Development Chair in Mathematics.

\section{Preliminaries on finite resolutions}\label{sec:prelim}

\begin{lemma} \label{lem:partial-resolution}
Let $R$ be a ring, let $M$ be a left $R$-module, let $m$ be a positive integer, and let
\begin{equation} \label{FirstProjResM}
C_{m-1} \xrightarrow{\partial_{m-1}} C_{m-2} \xrightarrow{\partial_{m-2}} \cdots \xrightarrow{\partial_1}
C_0 \xrightarrow{\partial_0} M \xrightarrow{} 0
\end{equation} 
be a (partial, augmented) resolution of $M$ by finitely generated projective $R$-modules.  
Then $M \in \FP_m$ if and only if $\ker \partial_{m-1}$ is finitely generated as an $R$-module.

\end{lemma}

\begin{proof}

If $\ker \partial_{m-1}$ is finitely generated, there exists a finite-rank free $R$-module
$C_m$ and an epimorphism $C_m\twoheadrightarrow \ker \partial_{m-1}$.  
Composing this epimorphism with the inclusion $\ker \partial_{m-1} \hookrightarrow C_{m-1}$ we get the resolution
\begin{equation*} 
C_m \xrightarrow{} C_{m-1} \xrightarrow{\partial_{m-1}} C_{m-2} \xrightarrow{\partial_{m-2}} \cdots \xrightarrow{\partial_1}
C_0 \xrightarrow{\partial_0} M \xrightarrow{} 0
\end{equation*} 
of $M$ by finitely generated projective $R$-modules, so $M\in\FP_m$ as desired.

Suppose now that $M\in\FP_m$, so it admits a (partial, augmented) projective resolution
\begin{equation} \label{SecProjResM}
P_m \xrightarrow{d_m} P_{m-1} \xrightarrow{d_{m-1}} P_{m-2} \xrightarrow{d_{m-2}} \cdots \xrightarrow{d_1}
P_0 \xrightarrow{d_0} M \xrightarrow{} 0
\end{equation}
by finitely generated (projective) $R$-modules. In particular, we have the finitely generated $R$-module
\[
\ker d_{m-1} = \im d_m.
\]
Generalized Schanuel's lemma
\cite[Corollary 5.5]{Lam} applied to \cref{FirstProjResM} and \cref{SecProjResM} provides us with finitely generated $R$-modules $U_1,U_2$ such that
\[
\ker \partial_{m-1} \oplus U_1\cong \ker d_{m-1} \oplus U_2.
\]
The right hand side is a finitely generated $R$-module, so $\ker \partial_{m-1}$, being a direct summand of a finitely generated $R$-module, is also a finitely
generated $R$-module, as required.
\end{proof}

\begin{corollary}
\label{cor:extend-partial-resolution}

Let $R$ be a ring, let $M$ be an $R$-module, let $m$ be a nonnegative integer with $M\in\FP_m$, and let $0 \leq r \leq m$.
Then every (augmented, partial)  resolution of $M$ by finite-rank free $R$-modules
\begin{equation*}
C_{r} \xrightarrow{\partial_{r}} C_{r-1} \xrightarrow{\partial_{r-1}} \cdots \xrightarrow{\partial_1}
C_0 \xrightarrow{\partial_0} M \xrightarrow{} 0
\end{equation*} 
can be extended to an augmented partial resolution of $M$ by finite-rank free $R$-modules
\begin{equation*}
C_m \xrightarrow{\partial_m}  \cdots \xrightarrow{\partial_r} C_{r} \xrightarrow{\partial_{r}} C_{r-1} \xrightarrow{\partial_{r-1}} \cdots \xrightarrow{\partial_1}
C_0 \xrightarrow{\partial_0} M \xrightarrow{} 0.
\end{equation*} 
\end{corollary}

\begin{proof}

We perform a descending induction on $r$, the case $r = m$ being obvious.
Suppose that $r < m$, so $M \in \FP_{r+1}$, hence it follows from \cref{lem:partial-resolution} that $\ker \partial_r$ is a finitely generated $R$-module.
Therefore there exists a finite-rank free $R$-module $C_{r+1}$ and a surjective homomorphism of $R$-modules $C_{r+1} \to \ker \partial_r$. Composing this with the inclusion of $\ker \partial_r$ into $C_r$ we get an augmented partial resolution
\begin{equation*}
C_{r+1} \xrightarrow{\partial_{r+1}} C_{r} \xrightarrow{\partial_{r}} C_{r-1} \xrightarrow{\partial_{r-1}} \cdots \xrightarrow{\partial_1}
C_0 \xrightarrow{\partial_0} M \xrightarrow{} 0
\end{equation*} 
of $M$ by finite-rank free $R$-modules, and conclude by applying the induction hypothesis.
\end{proof}

\begin{remark}

The argument in the proof can be used to complete the partial augmented free resolution of $M$ with the free $R$-modules in degrees exceeding $m$ not necessarily of finite rank.
This will be done silently in applications of \cref{cor:extend-partial-resolution}.

\end{remark}

For a set $X$ we denote by $F_X$ the free group on $X$.
For a ring $R$ and a set $I$, we have the functor $-^{(I)}$ associating to a left $R$-module $M$ the left $R$-module $M^{(I)}$ which is the direct sum of copies of $M$ indexed by $I$.
For a homomorphism of groups  $\pi \colon G \to H$ we denote by $\bar \pi \colon \Z G \to \Z H$ the ring homomorphism it induces on the associated group rings.
It endows $\Z H$ with the structure of a left $\Z G$-module.

The following is a consequence of the proof that finitely presented groups are $\FP_2$ and \cref{cor:extend-partial-resolution}.

\begin{lemma} \label{lem:prescribed}

Let $m \geq 3$ be an integer, let $X$ be a finite set, and let 
\[
1 \xrightarrow{} R \xrightarrow{} F_X \xrightarrow{\pi} Q \xrightarrow{} 1
\]
be a short exact sequence of groups with $Q \in \FP_m$ a finitely presented group.
Then there exists a finite $\Sigma \subseteq R$ generating $R$ as a normal subgroup of $F_X$, and a commutative diagram with exact rows
\[\begin{tikzcd}
	&& 0 & {\mathbb ZF_X^{(X)}} & {\mathbb Z F_X} & {\mathbb Z} & 0 \\
	{\mathbb Z Q^{(\dots)}} & \dots & {\mathbb ZQ^{(\Sigma)}} & {\mathbb ZQ^{(X)}} & {\mathbb ZQ} & {\mathbb Z} & 0
	\arrow[from=1-3, to=1-4]
	\arrow["{d_1^F}", from=1-4, to=1-5]
	\arrow["{\bar{\pi}^{(X)}}", from=1-4, to=2-4]
	\arrow["{\varepsilon^F}", from=1-5, to=1-6]
	\arrow["{\bar{\pi}}", from=1-5, to=2-5]
	\arrow[from=1-6, to=1-7]
	\arrow["{\mathrm{id}_\mathbb Z}", from=1-6, to=2-6]
	\arrow["{d_m^Q}", from=2-1, to=2-2]
	\arrow["{d_3^Q}", from=2-2, to=2-3]
	\arrow["{d_2^Q}", from=2-3, to=2-4]
	\arrow["{d_1^Q}", from=2-4, to=2-5]
	\arrow["{\varepsilon^Q}", from=2-5, to=2-6]
	\arrow[from=2-6, to=2-7]
\end{tikzcd}\]
the first (respectively, second) row consisting of finite-rank left $\Z F_X$-modules (respectively, $\Z Q$-modules), $\varepsilon^F$ being the augmentation map to the trivial module $\Z$, for every $x \in X$ we have $d_1^F(e_x) = x-1$.

\end{lemma}

\begin{proof}
    For $Q$, take a bouquet of circles indexed by $X$, add $2$-cells corresponding to the relations in $\Sigma$, and add $n$-cells for $n\ge 2$ to get an aspherical CW-complex. Then the embedding of the bouquet of $X$ circles induces the desired diagram.
\end{proof}

\begin{lemma} \label{lem:conjugation-lift}
Let $N$ be a group, let $\sigma\in\operatorname{Aut}(N)$, and consider an augmented free resolution
\[
\cdots\longrightarrow C_2\xrightarrow{d_2}C_1
 \xrightarrow{d_1}C_0\xrightarrow{d_0} \Z \xrightarrow{} 0.
\]
of the trivial \emph{left} $\Z N$-module. Then there is an augmented chain map of abelian groups $c_\sigma \colon C_*\to C_*$ such that for every $a \in N$, every nonnegative integer $q$, and every $x \in C_q$ we have
\[
c_\sigma(a x)= \sigma(a)c_\sigma(x).
\]
\end{lemma}

\begin{proof}

The forgetful functor from left $N$-modules to abelian groups is invariant under (precomposition by) the functor $M \mapsto M^\sigma$ of inflation by $\sigma$ from the category of left $N$-modules to itself, and the action of every $a \in N$ on each $m \in M^\sigma$ is the action of $\sigma(a)$ on $m$ viewed as an element of $M$. 
Our inflation functor is exact (and preserves free modules) so applying it to our resolution we get another such resolution.
By the comparison theorem \cite[Theorem~6.16]{Rotman}, the identity map of
$\Z$ lifts to a $\Z N$-linear augmented chain map
$
C_* \to C_*^\sigma,
$
so the required $c_\sigma$ is obtained by applying our forgetful functor.
\end{proof}


If $\sigma=\id_N$, one may choose $c_\sigma=\id_{C_*}$.


\section{Wall resolutions}\label{wall-resolutions}
\subsection{Wall resolutions of extensions}

We present here a variant of Wall's construction \cite{Wall}. Let
\[
1 \xrightarrow{}  N \xrightarrow{} H \xrightarrow{\nu}L \xrightarrow{} 1
\]
be a short exact sequence of groups. 
Let
\[
\cdots\xrightarrow{d_3^N} {\Z N}^{(X_2)} \xrightarrow{d_2^N} {\Z N}^{(X_1)}
\xrightarrow{d_1^N} \Z N^{(X_0)} \xrightarrow{\varepsilon^N} \Z \xrightarrow{} 0, \qquad X_0 = \{*\},
\]
be a free resolution of the trivial left $\Z N$-module $\Z$ with $\varepsilon^N$ the usual augmentation, and let
\[
\cdots \xrightarrow{d^L_3} {\Z L}^{(Y_2)} \xrightarrow{d_2^L} {\Z L}^{(Y_1)}
\xrightarrow{d_1^L} \Z L^{(Y_0)} \xrightarrow{\varepsilon^L} \Z \xrightarrow{} 0, \qquad Y_0 = \{*\},
\]
be a free resolution of the trivial left $\Z L$-module $\Z$ with $\varepsilon^L$ the usual augmentation.

For nonnegative integers $p,q,k$ consider the free left $\Z H$-modules
\[
W^H_{p,q}
 =\bigoplus_{y\in Y_p} \Z H\tensor_{\Z N} \Z N^{(X_q)}
 \cong \Z H^{(Y_p\times X_q)}, \qquad W^H_k=\bigoplus_{p+q=k}W^H_{p,q},
\]
note that $W_0^H = \Z H$, and equip it with its usual augmentation $\varepsilon^H$. 
We use the convention that $W_{p,q}^H = 0$ if (at least) one of $p,q$ is negative.

For integers $p, q$ let $d_{p,q}^0 \colon W_{p,q}^H \to W^H_{p,q-1}$ be the direct sum over $y \in Y_p$ of the base change (namely, extension of scalars from $\Z N$) to $\Z H$ of $d_q^N$, interpreted as the zero map if $p< 0$ or $q \leq 0$,
and set
\[
d_k^0 \colon W_k^H \to W_{k-1}^H, \qquad d^0_k = \bigoplus_{\substack{p+q = k}} d^0_{p,q}
\]
for every positive integer $k$.
These are homomorphisms of left $\Z H$-modules and for $k > 0$ we have 
\[
d^0_{k} \circ d^0_{k+1} = 0.
\]

The ring $\Z H$ is free as a right $\Z N$-module: a set of representatives for the cosets of $N$ in $H$ is a
right $\Z N$-basis. Therefore, for every nonnegative integer $p$,
tensoring our augmented $\Z N$-resolution with $\Z H$ over $\Z N$, and then taking the direct sum over $y\in Y_p$, gives an exact sequence of left $\Z H$-modules
\begin{equation} \label{eq:wall-vertical-augmentation}
\cdots \xrightarrow{d_{p,3}^0} W^H_{p,2} \xrightarrow{d^{0}_{p,2}} W^H_{p,1}
 \xrightarrow{d^0_{p,1}} W^H_{p,0}\xrightarrow{\bar{\nu}^{(Y_p)}} \Z L^{(Y_p)} \xrightarrow{} 0.
\end{equation}
Assuming $p$ is positive, we get the solid arrows of the following diagram of left $\Z H$-modules
\begin{equation} \label{ChainMapResol}
\begin{tikzcd}
	\cdots & {W^H_{p,2}} & {W^H_{p,1}} & {W^H_{p,0}} & {\mathbb Z L^{(Y_p)}} & 0 \\
	\cdots & {W^H_{p-1,2}} & {W^H_{p-1,1}} & {W^H_{p-1,0}} & {\mathbb Z L^{(Y_{p-1})}} & 0.
	\arrow["d^0_{p,3}", from=1-1, to=1-2]
	\arrow["d^0_{p,2}", from=1-2, to=1-3]
	\arrow["{f_{p,2}^H}", dotted, from=1-2, to=2-2]
	\arrow["d^0_{p,1}", from=1-3, to=1-4]
	\arrow["{f_{p,1}^H}", dotted, from=1-3, to=2-3]
	\arrow["{\bar{\nu}^{(Y_p)}}", from=1-4, to=1-5]
	\arrow["{f_{p,0}^H}", dotted, from=1-4, to=2-4]
	\arrow[from=1-5, to=1-6]
	\arrow["{d_p^L}", from=1-5, to=2-5]
	\arrow["d^0_{p-1,3}", from=2-1, to=2-2]
	\arrow["d^0_{p-1,2}", from=2-2, to=2-3]
	\arrow["d^0_{p-1,1}", from=2-3, to=2-4]
	\arrow["{\bar{\nu}^{(Y_{p-1})}}", from=2-4, to=2-5]
	\arrow[from=2-5, to=2-6]
\end{tikzcd}
\end{equation}

Choose any homomorphism $f_{p,0}^H \colon W^H_{p,0} \to W^H_{p-1,0}$ of left $H$-modules that makes the rightmost square above commute - one exists because $W^H_{p,0}$ is a free left $\Z H$-module and $\bar{\nu}^{(Y_{p-1})}$ is surjective.
Continuing in this fashion, and using diagram chasing to check the necessary inclusion of images, we can produce the dotted arrows, thus obtaining a morphism of chain complexes of left $\Z H$-modules.  

For integers $p, q$ let $d_{p,q}^1 \colon W_{p,q}^H \to W^H_{p-1,q}$ be the homomorphism $(-1)^q f_{p,q}^H$ of left $H$-modules, interpreted as the zero map if $p < 1$ or $q < 0$, so that for every positive integer $k$ we have the homomorphism
\[
d^1_k \colon W_k^H \to W_{k-1}^H, \qquad d^1_k = \bigoplus_{p+q = k} d^1_{p,q}
\]
of left $H$-modules.
From the chain map condition, for every positive integer $k$ we get that
\[
d^0_{k} \circ d^1_{k+1} + d^1_{k} \circ d^0_{k+1}=0.
\]

\begin{lemma} \label{Wall}

For every integer $i \geq 2$ and integers $p, q$ there exists a homomorphism 
\[
d^i_{p,q} \colon W^H_{p,q} \to W^H_{p-i,q+i-1}
\]
of left $H$-modules such that for the homomorphisms of left $H$-modules given for a positive integer $k$ by
\[
d_k^i \colon W^H_k \to W^H_{k-1}, \qquad d_k^i = \bigoplus_{p+q = k} d^i_{p,q}, \qquad
d_k^H \colon W^H_k \to W^H_{k-1}, \qquad d_k^H = \sum_{i=0}^k d_k^i,
\]
we get an augmented resolution of the trivial left $\Z H$-module $\Z$ by free left $\Z H$-modules 
\begin{equation} \label{PurportedResolution}
\cdots \xrightarrow{d^H_3} W^H_2 \xrightarrow{d^H_2} W^H_1 \xrightarrow{d^H_1} W^H_{0} \xrightarrow{\varepsilon^H} \Z \xrightarrow{} 0.
\end{equation}

\end{lemma}

\begin{proof}

We construct by induction on $r$ the maps $d^0_{p,q}, \dots, d^r_{p,q}$ for $p,q \in \Z$ in such a way that they satisfy
\begin{equation} \label{IndDemandDiff}
\sum_{i=0}^r d^{i}_{p-r+i,q+r-i-1} \circ d^{r-i}_{p,q} = 0
\end{equation}
with $r=0$ and $r=1$ being the base case handled above.
Suppose now that $r \geq 2$ and for $p,q \in \Z$ set
\[
A^r_{p,q} \colon W^H_{p,q} \to W^H_{p-r,q+r-2}, \qquad
A^r_{p,q} =\sum_{i=1}^{r-1}d^{i}_{p-r+i,q+r-i-1} \circ d^{r-i}_{p,q}
\]
so that \cref{IndDemandDiff} is equivalent to 
\begin{equation} \label{EquivIndDemandDiff}
d^0_{p-r,\,q+r-1}\circ d^r_{p,q}
+
 d^r_{p,q-1}\circ d^0_{p,q}
+
 A^r_{p,q}
=0.
\end{equation} 

We claim that for every $p,q \in \Z$ we have
\begin{equation*}
d^{0}_{p-r,q+r-2} \circ A^r_{p,q} = A^r_{p,q-1} \circ d^{0}_{p,q}.
\end{equation*}
By the definition of \(A^r_{p,q}\) and \(A^r_{p,q-1}\) the difference we want to show is $0$ equals
\begin{equation*}
d^0_{p-r,\,q+r-2}\circ A^r_{p,q}
- A^r_{p,q-1}\circ d^0_{p,q} \\
=\sum_{i=1}^{r-1}\Bigl(
 d^0_{p-r,\,q+r-2}
 \circ d^i_{p-r+i,\,q+r-i-1}
 \circ d^{r-i}_{p,q}
 -d^i_{p-r+i,\,q+r-i-2}
 \circ d^{r-i}_{p,q-1}
 \circ d^0_{p,q}
\Bigr).
\end{equation*}
For each \(1\leq i\leq r-1\), apply \cref{IndDemandDiff} inductively with indices \(p-r+i\) and \(q+r-i-1\) to get
\[
d^0_{p-r,\,q+r-2}
 \circ d^i_{p-r+i,\,q+r-i-1} 
= -d^i_{p-r+i,\,q+r-i-2}
 \circ d^0_{p-r+i,\,q+r-i-1} 
 -\sum_{a=1}^{i-1}
 d^a_{p-r+a,\,q+r-a-2}
 \circ d^{i-a}_{p-r+i,\,q+r-i-1}.
\]
Substituting this into the previous displayed equation expresses the difference we want to show is $0$ as
\[
-\sum_{i=1}^{r-1}
 d^i_{p-r+i,\,q+r-i-2}
 \circ
 \Bigl(
 d^0_{p-r+i,\,q+r-i-1}\circ d^{r-i}_{p,q}
 +d^{r-i}_{p,q-1}\circ d^0_{p,q}
 \Bigr) -\sum_{i=1}^{r-1}\sum_{a=1}^{i-1}
 d^a_{p-r+a,\,q+r-a-2}
 \circ d^{i-a}_{p-r+i,\,q+r-i-1}
 \circ d^{r-i}_{p,q}.
\]
Now apply \cref{IndDemandDiff} inductively for $r-i$ to get that
\[
d^0_{p-r+i,\,q+r-i-1}\circ d^{r-i}_{p,q}
+d^{r-i}_{p,q-1}\circ d^0_{p,q} 
= -\sum_{b=1}^{r-i-1}
 d^b_{p-r+i+b,\,q+r-i-b-1}
 \circ d^{r-i-b}_{p,q}.
\]
Substituting this into the previous displayed equation we get that the to be $0$ difference equals
\[
\sum_{i=1}^{r-1}\sum_{b=1}^{r-i-1}
 d^i_{p-r+i,\,q+r-i-2}
 \circ d^b_{p-r+i+b,\,q+r-i-b-1}
 \circ d^{r-i-b}_{p,q} 
 -\sum_{i=1}^{r-1}\sum_{a=1}^{i-1}
 d^a_{p-r+a,\,q+r-a-2}
 \circ d^{i-a}_{p-r+i,\,q+r-i-1}
 \circ d^{r-i}_{p,q}.
 \]
The first double sum is 
\[
\sum_{\substack{u,v,w \geq 1 \\ u+v+w = r}} d^u_{p-r+u,\,q+r-u-2}
 \circ d^v_{p-r+u+v,\,q+r-u-v-1}
 \circ d^w_{p,q}
\]
as we see by putting \(u=i\), \(v=b\), \(w=r-i-b\), and so is the second double sum -- put \(u=a\), \(v=i-a\), and
\(w=r-i\). 
Hence the two double sums cancel each other so our claim is proven.

If \(p<r\) or
\(q<0\), take \(d^r_{p,q}\) to be the zero map. Now fix \(p\geq r\), and induct on $q$.
In the base case $q=0$ we need to define $d^r_{p,0}$ in such a way that \cref{EquivIndDemandDiff} holds, which in view of $d^0_{p,0} = 0$ becomes
\[
d^0_{p-r,r-1} \circ d^r_{p,0} = -A^r_{p,0}.
\]
Because $W^H_{p,0}$ is a free left $\Z H$-module this can be done if (and only if) the image of $A^r_{p,0}$ is contained in that of $d^0_{p-r, r-1}$.
If $r = 2$, by exactness of \cref{eq:wall-vertical-augmentation} the latter image is $\ker \bar{\nu}^{(Y_{p-2})}$, and we have 
\[
\bar{\nu}^{(Y_{p-2})} \circ A^r_{p,0} = \bar{\nu}^{(Y_{p-2})} \circ d^1_{p-1,0} \circ d^1_{p,0} = 
\bar{\nu}^{(Y_{p-2})}  \circ f_{p-1,0}^H \circ f_{p,0}^H = d^L_{p-1} \circ \bar{\nu}^{(Y_{p-1})}  \circ f_{p,0}^H = d^L_{p-1} \circ d^L_p \circ \bar{\nu}^{(Y_{p})}  = 0
\]
by the commutativity of \cref{ChainMapResol}, as required.
Now if $r > 2$, the claim we have proven says that $d^0_{p-r,r-2} \circ A^r_{p,0} = 0$ namely the image of $A^r_{p,0}$ is contained in $\ker d^0_{p-r,r-2}$, which is $\im d^0_{p-r, r-1}$ by the exactness of \cref{eq:wall-vertical-augmentation}, as required.

Suppose now that $q \geq 1$.
In view of \cref{EquivIndDemandDiff}, in order to construct $d^r_{p,q}$ it is (necessary and) sufficient to check that the image of $A^{r}_{p,q} + d^r_{p,q-1} \circ d^0_{p,q}$ is contained in that of $d^0_{p-r, q+r-1}$.
By exactness of \cref{eq:wall-vertical-augmentation}, the latter image is $\ker d^0_{p-r,q+r-2}$ and we get from our claim and \cref{EquivIndDemandDiff} for $q-1$ that 
\[
\begin{split}
d^0_{p-r,q+r-2} \circ (A^{r}_{p,q} + d^r_{p,q-1} \circ d^0_{p,q}) &= A^r_{p,q-1} \circ d^{0}_{p,q} + d^0_{p-r,q+r-2} \circ d^r_{p,q-1} \circ d^0_{p,q} \\
&= A^r_{p,q-1} \circ d^{0}_{p,q} - (A^r_{p,q-1} + d^r_{p,q-2} \circ d^0_{p,q-1}) \circ d^0_{p,q} = 0
\end{split}
\]
because $d^0_{p,q-1} \circ d^0_{p,q} = 0$, as required.

The construction of the maps $d^{r}_{p,q}$ is thus complete, and it follows from \cref{IndDemandDiff} that for every positive integer $k$ we have $d_k^H \circ d_{k+1}^H = 0$.
Therefore, to conclude that \cref{PurportedResolution} is a complex it remains to check that $\varepsilon^H \circ d^H_1 = 0$.
Indeed we have
\[
\begin{split}
\varepsilon^H \circ d^H_1 &= \varepsilon^L \circ \bar{\nu}^{(Y_{0})}  \circ (d^0_1 + d^1_1) = \varepsilon^L \circ \bar{\nu}^{(Y_{0})}  \circ (d^0_{0,1} + d^1_{1,0}) = \varepsilon^L \circ \bar{\nu}^{(Y_{0})}  \circ d^0_{0,1} + \varepsilon^L \circ \bar{\nu}^{(Y_{0})}  \circ d^1_{1,0} = \varepsilon^L \circ \bar{\nu}^{(Y_{0})}  \circ d^1_{1,0} \\
&= \varepsilon^L \circ \bar{\nu}^{(Y_{0})}  \circ f_{1,0}^H =\varepsilon^L \circ d^L_1 \circ \bar{\nu}^{(Y_{1})}  = 0.
\end{split}
\]

Our last task is to show that \cref{PurportedResolution} is exact.
For every $k > 0$ we have a finite increasing filtration
\[
\bigoplus_{\substack{p+q = k \\ p \leq a}} W^H_{p,q}, \qquad 0 \leq a \leq k
\]
exhausting $W^H_k$ by free left $\Z H$-modules.
The differential $d^H_k$ respects the filtration, and the associated graded differential is induced by $d^0_k$.
In view of the exactness of \cref{eq:wall-vertical-augmentation} the first page of the associated spectral sequence is
\[
E_{p,q}^1 = 
\begin{cases}
\Z L^{(Y_p)} &q = 0 \\
0 &q>0
\end{cases}
\]
with differential on the row $q=0$ induced by $d^1_{p,0} = f_{p,0}$. It follows from the commutativity of the diagram in \cref{ChainMapResol} that this differential on the row $q = 0$ is $d^L_p$.

We conclude that the second page of the spectral sequence is given by
\[
E^2_{p,q} = 
\begin{cases}
\Z &(p,q )= (0,0) \\
0 &(p,q) \neq (0,0)
\end{cases}
\]
so the spectral sequence collapses on this page, and converges to the homology of the complex $(W^H_*, d^H_*)$.
We thus obtain the identification
\[
H_k(W^H_*, d^H_*) = 
\begin{cases}
\Z &k = 0 \\
0 &k > 0
\end{cases}
\]
which for $k=0$ is induced by $\varepsilon^H$ because it is the composition
$
W^H_{0,0}=\mathbb ZH
\xrightarrow{\bar{\nu} }
\mathbb ZL
\xrightarrow{\varepsilon^L}
\mathbb Z.
$
The required exactness of \cref{PurportedResolution} follows.
\end{proof}

\subsection{Wall resolutions of fibre products}
Fix a positive integer $n$, a finite set $X$, and exact sequences of groups
\begin{equation} \label{TwoExactSequencesFreeFact}
1 \xrightarrow{} N \xrightarrow{\eta} G \xrightarrow{\phi} Q \xrightarrow{} 1, \qquad 
1 \xrightarrow{} R \xrightarrow{\mu} F_X \xrightarrow{\pi} Q \xrightarrow{} 1.
\end{equation}
Then for the fibre product
$
P=G\times_Q F_X
$
we have the following commutative diagram of groups with exact rows, columns and diagonal
\[\begin{tikzcd}
	1 && 1 & 1 & \\
	& {N\times R} & {R} & {R} \\
	1 & {N} & {P} & {F_X} & 1 \\
	1 & {N} & {G} & Q & 1 \\
	&& 1 & 1 & 1
	\arrow[from=1-3, to=2-3]
	\arrow[from=1-1, to=2-2]
	\arrow[from=1-4, to=2-4]
	\arrow[from=2-2, to=3-3]
	\arrow[equal, from=2-3, to=2-4]
	\arrow["{(1,\mu)}", from=2-3, to=3-3]
	\arrow["\mu", from=2-4, to=3-4]
	\arrow[from=3-1, to=3-2]
	\arrow["{(\eta,1)}"', from=3-2, to=3-3]
	\arrow[equal, from=3-2, to=4-2]
	\arrow["\mathrm{pr}_2", from=3-3, to=3-4]
	\arrow["\mathrm{pr}_1", from=3-3, to=4-3]
	\arrow[from=3-3, to=4-4]
	\arrow[from=3-4, to=3-5]
	\arrow["\pi", from=3-4, to=4-4]
	\arrow[from=4-1, to=4-2]
	\arrow["\eta"', from=4-2, to=4-3]
	\arrow["\phi"', from=4-3, to=4-4]
	\arrow[from=4-3, to=5-3]
	\arrow[from=4-4, to=4-5]
	\arrow[from=4-4, to=5-4]
	\arrow[from=4-4, to=5-5]
\end{tikzcd}\]
Our assumptions are
\[
N\in\FP_n,
\qquad G\in\FP_{n+1},
\qquad Q\in\FP_{n+2},
\qquad Q\text{ finitely presented}.
\]

Our assumption that $N \in \FP_n$ in conjunction with \cref{cor:extend-partial-resolution} applied to the ring $\Z N$, $m=n$, $r=0$, and the augmentation $\Z N \to \Z=M$
provide us with a free resolution of the trivial left $\Z N$-module
\begin{equation} \label{Nresol}
\cdots\xrightarrow{d_3^N} {\Z N}^{(X_2)} \xrightarrow{d_2^N} {\Z N}^{(X_1)}
\xrightarrow{d_1^N} \Z N^{(X_0)} \xrightarrow{\varepsilon^N} \Z \xrightarrow{} 0, \qquad X_0 = \{*\},
\end{equation}
with the set $X_q$ being finite for $0 \leq q \leq n$.

Our assumption that $Q \in \FP_{n+2}$ in conjunction with \cref{lem:prescribed} applied to $m=n+2$ provide us with a finite $\Sigma \subseteq R$ generating it as a normal subgroup of $F_X$ and a commutative diagram 
\begin{equation} \label{ResolFQ}
\begin{tikzcd}
	&& 0 & {\mathbb ZF_X^{(X)}} & {\mathbb Z F_X} & {\mathbb Z} & 0 \\
	& \dots & {\mathbb ZQ^{(\Omega_2)}} & {\mathbb ZQ^{(\Omega_1)}} & {\mathbb ZQ^{(\Omega_0)}} & {\mathbb Z} & 0
	\arrow[from=1-3, to=1-4]
	\arrow["{d_1^F}", from=1-4, to=1-5]
	\arrow["{\bar{\pi}^{(X)}}", from=1-4, to=2-4]
	\arrow["{\varepsilon^F}", from=1-5, to=1-6]
	\arrow["{\bar{\pi}}", from=1-5, to=2-5]
	\arrow[from=1-6, to=1-7]
	\arrow["{\mathrm{id}_\mathbb Z}", from=1-6, to=2-6]
	\arrow["{d_3^Q}", from=2-2, to=2-3]
	\arrow["{d_2^Q}", from=2-3, to=2-4]
	\arrow["{d_1^Q}", from=2-4, to=2-5]
	\arrow["{\varepsilon^Q}", from=2-5, to=2-6]
	\arrow[from=2-6, to=2-7]
\end{tikzcd}
\end{equation}
with exact rows, $\Omega_0 = \{*\}$, $\Omega_1 = X$, $\Omega_2 = \Sigma$, the set $\Omega_t$ being finite for $0 \leq t \leq n+2$.

By \cref{Wall} we have a resolution of the trivial left $\Z P$-module $\Z$ by free left $\Z P$-modules 
\begin{equation} \label{ResolutionP}
\cdots \xrightarrow{d^P_3} W^P_2 \xrightarrow{d^P_2} W^P_1 \xrightarrow{d^P_1} W^P_{0} \xrightarrow{\varepsilon^P} \Z \xrightarrow{} 0, \qquad W_k^P = W^P_{0,k} \oplus W_{1,k-1}^P = \Z P^{(X_k)} \oplus \Z P^{(X \times X_{k-1})}
\end{equation}
arising from the exact sequence $1\to N\to P\to F_X\to 1$,
\cref{Nresol} and the upper row in \cref{ResolFQ}.
%
%

Next we construct using \cref{Wall} a resolution of the trivial left $\Z G$-module $\Z$ by free left $\Z G$-modules 
\begin{equation} \label{ResolutionG}
\cdots W^G_1 \xrightarrow{d^G_1} W^G_{0} \xrightarrow{\varepsilon^G} \Z \xrightarrow{} 0, \quad W_{p,q}^G = \Z G^{(\Omega_p \times X_q)}, \quad W_{0,q}^G \cong \Z G \otimes_{\Z P} W_{0,q}^P, \quad W_{1,q}^G \cong \Z G \otimes_{\Z P} W_{1,q}^P
\end{equation}
arising from the exact sequence $1\to N\to G\to Q\to 1$,
\cref{Nresol} 
and the lower row in \cref{ResolFQ}, the map $\bar{\mathrm{pr}}_1 \colon \Z P \to \Z G$ endowing $\Z G$ with the structure of a right $\Z P$-module, and for every integer $q$ put
\[
f^G_{1,q} = \Z G \otimes_{\Z P} f^P_{1,q}.
\]

In order to be able to apply \cref{Wall} we need to check that the diagram obtained by putting $p=1$, $H = G$, $L = Q$, $Y_0 = \{*\}$, and $Y_1 = X$ in \cref{ChainMapResol} is commutative. This follows by applying the functor $\Z G \otimes_{\Z P} - $ to \cref{ChainMapResol} with $p=1$, $H = P$, $L = F_X$, and using \cref{ResolFQ} in conjunction with the fact that the functor $\Z G \otimes_{\Z P} - $ transforms the homomorphism of $P$-modules $\bar{\mathrm{pr}}_2 \colon \Z P \to \Z F_X$ to the homomorphism of $G$-modules $\bar \phi \colon \Z G \to \Z Q$. 
To see this note that we have an isomorphism 
\[
\Z G \otimes_{\Z P} \Z F_X \xrightarrow{\sim} \Z Q, \qquad \alpha \otimes \beta \mapsto \bar \phi(\alpha) \bar \pi(\beta), \quad \alpha \in \Z G, \beta \in \Z F_X
\]
of left $G$-modules.
Therefore \cref{Wall} applies and provides us with \cref{ResolutionG}.

As a result of our choice of $f^G_{1,q}$,
we get a homomorphism $\Pi \colon W^P_* \to W^G_*$ of augmented resolutions of the trivial left $\Z P$-module $\Z$ given, for every $p \in \{0,1\}$ and every integer $q$, by the surjective homomorphism $\Pi_{p,q}$ of left $\Z P$-modules 
\[
W_{p,q}^P \cong \Z P \otimes_{\Z P}  W_{p,q}^P \xrightarrow{\bar{\mathrm{pr}}_1 \otimes \mathrm{id}} \Z G \otimes_{\Z P} W^P_{p,q} \cong W^{G}_{p,q}.
\]

\section{Truncating the Wall resolution}\label{truncation}

\begin{lemma} \label{lem:normalizer-chain-homotopy}

Fix $s\in P$, and for $a\in N$ put $\sigma(a)=s^{-1}as$. 
Choose an augmented chain map $c$ of abelian groups as in
\cref{lem:conjugation-lift} for the augmented resolution in \cref{Nresol}.
Then we have a $\Z P$-linear chain map
\[
T_q \colon W_{0,q}^P \cong \Z P \otimes_{\Z N} \Z N^{(X_q)} \xrightarrow{\cdot s \otimes c} \Z P \otimes_{\Z N} \Z N^{(X_q)} \cong W_{0,q}^P, \qquad q \geq 0,
\]
inducing on homology in degree $0$, which we identify with $\Z F_X$ using \cref{eq:wall-vertical-augmentation} with $H = P$, $p=0$, $L = F_X$, and $\nu = \mathrm{pr}_2$, right multiplication by $\mathrm{pr}_2(s)$.

Denoting by $\iota$ the $\Z P$-linear inclusion of chain complexes
\[
\iota_q \colon W_{0,q}^P  \to W_{q}^P, \qquad q \geq 0,
\]
we moreover get that $\iota \circ T$ is $\Z P$-linearly chain homotopic to $\iota$.

\end{lemma}

\begin{proof}

To see that \(T_q\) is well-defined note that for \(u\in \mathbb ZP\), \(a\in N\), and \(x\in \Z N^{(X_q)} \) we have
\[
T_q(ua\otimes x)= uas\otimes c_q(x) = us\sigma(a) \otimes c_q(x) = us \otimes \sigma(a)c_q(x) = us \otimes c_q(ax) = T_q(u \otimes ax).
\]
Clearly, $T_q$ is additive, and it is moreover \(\mathbb ZP\)-linear because for each \(b\in \mathbb ZP\) we have
\[
T_q\bigl(b(u\otimes x)\bigr)
=
T_q(bu\otimes x)
=
bus\otimes c_q(x)
=
bT_q(u\otimes x).
\]
To see that \(T\) is a chain map note that for every positive integer $q$ we have
\[
d^0_{0,q}T_q(u\otimes x)
=
d^0_{0,q}(us\otimes c_q(x))  
=
us\otimes d_q^N c_q(x)  
=
us\otimes c_{q-1}d_q^N(x)  
=
T_{q-1}(u\otimes d_q^N(x))  
=
T_{q-1}d^0_{0,q}(u\otimes x).
\]

With the identifications $W_{0,0}^P \cong \Z P \otimes_{\Z N} \Z N^{(X_0)} \cong \Z P \otimes_{\Z N} \Z N $ we are making, the augmentation map to $\Z F_X \cong \Z F_X \otimes_{\Z} \Z$ is $\bar{\mathrm{pr}}_2 \otimes \varepsilon^N$,
so to see that $T$ induces right multiplication by $\mathrm{pr}_2(s)$ note that
\[
(\bar{\mathrm{pr}}_2 \otimes \varepsilon^N) \circ T_0(u\otimes x) =
\bar{\operatorname{pr}}_2(us) \otimes \varepsilon^N(c_0(x)) =
\bar{\operatorname{pr}}_2(u)\operatorname{pr}_2(s) \otimes \varepsilon^N(x) =
(\bar{\mathrm{pr}}_2(u) \otimes \varepsilon^N(x)) \cdot \operatorname{pr}_2(s),
\]
where $u \in \Z P$, $x \in \Z N$, and we have used the fact that \(c\) is an augmented chain map.

The map on degree $0$ homology induced by $\iota$ is the augmentation $\varepsilon^{F_X} \colon \Z F_X \to\Z$, which is invariant under right multiplication by $\mathrm{pr}_2(s)$ so $\iota$ and $\iota \circ T$ induce the same map on $H_0$.
The $\Z P$-linear chain homotopicity of $\iota$ and $\iota \circ T$ is now a consequence of the comparison theorem \cite[Theorem~6.16]{Rotman} applied to the chain maps $\iota$ and $\iota \circ T$ from the augmented chain complex in \cref{eq:wall-vertical-augmentation} to the augmented chain complex in \cref{ResolutionP} because these are projective resolutions.
\end{proof}

Fix a finite symmetric generating set $S$ of $P$.  
We put 
$
\Delta = (1,\mu)(\Sigma)
$
and note that $\Delta$ commutes with $(\eta,1)(N)$ in $P$.

For $H \in \{G,P\}$ and every $x \in X_{n+1}$ consider the projection onto the indexed by $x$ factor
\[
\mathrm{pr}_x \colon W^H_{n+1} \to \Z H
\]
obtained by projecting first to the summand $W^H_{0,n+1} \cong \Z H^{(X_{n+1})}$ and then to the $x$th coordinate.
For every $s\in S \cup \Delta$, choose $c^s$ and a $\Z P$-linear chain homotopy $K^s$ between $\iota$ and $\iota \circ T^s$ as in
\cref{lem:normalizer-chain-homotopy}, taking
$c^s = \id$ if $s \in \Delta$. We define the sets
\[
\mathcal A_s = \{x \in X_{n+1} : \textup{there exists } w \in W^P_{0,n} \textup{ such that } \mathrm{pr}_x (K_n^s(w)) \neq 0 \}, \qquad \mathcal A=\bigcup_{s\in S \cup \Delta}\mathcal A_s,
\]
and note that these are finite because $W^P_{0,n}$ is a finitely generated $\Z P$-module, and $S$, $\Delta$ are finite.

For $H \in \{P,G\}$, define an augmented subcomplex $A^H_* \subseteq W^H_*$ of finite-rank free $\Z H$-submodules by
\[
A^H_k =
\begin{cases}
W^H_k & 0 \leq k \leq n\\
\bigoplus_{\substack{p+q=n+1\\p\geq1}}W^H_{p,q}
  \oplus \Z H^{(\mathcal A)} &k = n+1 \\
0 &k > n+1.
\end{cases}
\]
It is exact (at $\Z$ and) in degrees up to $n-1$ (inclusive), we have $\Pi(A^P_*) \subseteq A^G_*$, and
$
K^s_n(W^P_{0,n}) \subseteq A^P_{n+1}
$
for every $s\in S \cup \Delta$.
Define the homomorphism of left $\Z N$-modules
\[
j \colon \ker d^N_n \longrightarrow H_n(A^P_*), 
\qquad j(z)=[1\tensor z].
\]

\begin{proposition} \label{prop:normal-centralizer}

The map $j$ is surjective, and $(1, \mu)(R)$ acts trivially on $H_n(A^P_*)$.

\end{proposition}

\begin{proof}

For \(s\in S\cup\Delta\) we have
$
\iota_nT^s_n - \iota_n = d^P_{n+1} \circ K^s_n + K^s_{n-1} \circ d^0_{0,n}
$
so for \(z\in \ker d_n^N\) we get that
\[
\begin{split}
\iota_nT^s_n(1\otimes z)-\iota_n(1\otimes z)
=
d^P_{n+1} (K^s_n(1\otimes z)) + K^s_{n-1} (d^0_{0,n}(1 \otimes z)) &= 
d^P_{n+1} (K^s_n(1\otimes z)) + K^s_{n-1} (1 \otimes d_n^N(z)) \\
&= d^P_{n+1} (K^s_n(1\otimes z)) \in d^P_{n+1}(A^P_{n+1})
\end{split}
\]
hence 
$
[T^s_n(1\otimes z)]
=
[1\otimes z]
$
as classes in $H_n(A^P_*)$. 
On the other hand
$
T^s_n(1\otimes z)
=
s\otimes c^s_n(z)
$
so
\begin{equation} \label{sActionOnj}
s\cdot j(c^s_n(z))
=
[s\otimes c^s_n(z)]
=
[1\otimes z]
=
j(z).
\end{equation}
It follows that 
$
\im j \subseteq s \cdot \im j.
$
Since $S$ is symmetric it follows that 
$
\im j = s \cdot \im j.
$
Because $S$ generates $P$ we infer that $\im j$ is a $P$-submodule of $H_n(A^P_*)$. 
Plugging \(s=\delta\) in \cref{sActionOnj} we see that $\delta$ acts as the identity on $\im j$, so the normal subgroup of $P$ generated by $\Delta$, which is $(1, \mu)(R)$ because $R$ is the normal subgroup of $F_X$ generated by $\Sigma$, acts trivially on $\im j$. 



In order to show that $\im j = H_n(A^P_*)$, take
$
[\alpha]\in H_n(A^P_*)
$
represented by
\[
\alpha \in \ker d_n^P \subseteq W^P_n = W^P_{0,n} \oplus W^P_{1,n-1}.
\]
Because \cref{ResolutionP} is exact, there exists $\beta \in W^P_{n+1}$ with $d_{n+1}^P(\beta)=\alpha$.
Writing $\beta=\beta_0+b$ with $\beta_0 \in W^P_{0,n+1}$ and
$
b \in W^P_{1,n} \subseteq A^P_{n+1}
$
, we have in $H_n(A^P_*)$ 
\[
[\alpha] = [\alpha - d^P_{n+1}(b)] = [d_{n+1}^P(\beta_0)].
\]
Now, the element \(d_{n+1}^P(\beta_0) \in W^P_{0,n}\) lies in $\ker d^P_n$, so
\[
d_{n+1}^P(\beta_0) \in \ker (d^P_n|_{W_{0,n}^P}) = \ker (\mathrm{id}_{\mathbb ZP}\otimes d_n^N) = \mathbb ZP\otimes_{\mathbb ZN}\ker d_n^N
\]
as \(\mathbb ZP\) is free as a right \(\mathbb ZN\)-module.

We can thus express \(d_{n+1}^P(\beta_0)\) as a finite sum of elements of the form
$
h\otimes z
$
with $h \in\mathbb ZP$ and $z \in \ker d_n^N$.
In $H_n(A^P_*)$ we have
$
[h\otimes z]
=
h\cdot [1\otimes z]
=
h\cdot j(z)
$
, so \([\alpha]\) lies in the left \(\mathbb ZP\)-submodule of $H_n(A^P_*)$ generated by $\im j$, which is just $\im j$ as we have seen that it is a $P$-submodule of $H_n(A^P_*)$.
\end{proof}

\subsection{Finite generation of the kernel}

\begin{proposition}\label{KerFinGen}
The homomorphism of left $\Z P$-modules
$
H_n(A^P_*) \xrightarrow{} H_n(A^G_*)
$
induced by $\Pi$ is surjective and its kernel is finitely generated as a left \(\mathbb ZP\)-module.
\end{proposition}

\begin{proof}

For \(H\in\{P,G\}\) inherited from $W^H_*$ is a finite increasing filtration by $\Z P$-subcomplexes
\[
F_aA_k^H
=A^H_k\cap
\bigoplus_{\substack{p+q=k\\ p\leq a}} W^H_{p,q}, \qquad 0 \leq a \leq k, \qquad k \geq 0,
\]
%
with associated spectral sequence
$
\partial^r_{p,q}(H) \colon E^r_{p,q}(H)\xrightarrow{} E^r_{p-r,q+r-1}(H)
$
of $\Z P$-modules.
The map
$
\Pi
$
preserves our filtrations so it induces a morphism of spectral sequences of $\Z P$-modules
$
E^r(P) \xrightarrow{} E^r(G).
$

Under the identifications
\[
E^1_{1,0}(P)\cong \Z F_X^{(X)}, \quad
E^1_{0,0}(P)\cong \Z F_X,
\]
we have
$
\partial^1_{1,0}(P) = d^F_1
$
which is injective. Therefore, if $n=1$, then $E^2_{1,n-1}(P) = 0$ by the exactness of \cref{eq:wall-vertical-augmentation} for $H = P$, $p=1$. If $n\ge 2$, then $E^2_{1,n-1}(P)=0$ since $E^1_{1,n-1}(P)=0$.
We conclude that the only (possibly) nonzero term in total degree $n$ is $E^2_{0,n}(P)$ with the differentials entering and leaving it being zero.

As a result
\[
H_n(A^P_*)\cong E^2_{0,n}(P).
\]

For \(H\in\{P,G\}\), since \(\Z H\) is free as a right \(\Z N\)-module, we have
\[
E^1_{0,n}(H)
\cong
\frac{
\Z H\otimes_{\Z N}\ker(d^N_n)
}{
d^0_{0,n+1}\bigl(\Z H^{(\mathcal A)}\bigr)
}, \qquad
E^1_{1,n}(H)
\cong
\Z H^{(X)}
\otimes_{\Z N}
\ker(d^N_n),
\]
so our choice
$
f^G_{1,n}=\Z G\otimes_{\Z P}f^P_{1,n}
$
implies that
\[
E^1_{0,n}(G)\cong \Z G\otimes_{\Z P}E^1_{0,n}(P),
\qquad
E^1_{1,n}(G)\cong \Z G\otimes_{\Z P}E^1_{1,n}(P),
\]
and under these identifications
$
\partial^1_{1,n}(G)= \Z G\otimes_{\Z P} \partial^1_{1,n}(P).
$
Since tensor product is right exact we have
\[
E^2_{0,n}(G)
\cong
\operatorname{coker} \partial^1_{1,n}(G) \cong
\Z G\otimes_{\Z P}E^2_{0,n}(P) \cong  \Z G\otimes_{\Z P}H_n(A^P_*) \cong H_n(A_*^P)
\]
where the last isomorphism is a consequence of \cref{prop:normal-centralizer} because $(1, \mu)(R) = \ker \mathrm{pr}_1$.

Note that \(E^2_{p,n-p}(G) = 0\) for \(p>0\): if \(p<n\) this follows from exactness of \cref{eq:wall-vertical-augmentation} with $H = G$, while if \(p=n\) it is a consequence of the exactness of \cref{ResolFQ} at \(\Z Q^{(\Omega_n)}\). Therefore
\[
H_n(A^G_*)\cong E^\infty_{0,n}(G).
\]
By naturality of the spectral sequence, the map
$
H_n(A^P_*)\to H_n(A^G_*)
$
induced by $\Pi$ is (identified with) the edge map
$
E^2_{0,n}(G)\xrightarrow{} E^\infty_{0,n}(G)
$
that exists, and is surjective, because the outgoing differentials from the \(p=0\) column of the spectral sequence $E^r_{p,q}(G)$ vanish.
It remains to show that the kernel of this edge map is finitely
generated as a \(\Z P\)-module.

The maps
\[
E^2_{0,n}(G)\to E^3_{0,n}(G)\to \cdots \to E^\infty_{0,n}(G)
\]
are successive quotients by the images of the incoming differentials
\[
\partial^r_{r, n-r+1}(G) \colon E^r_{r,n-r+1}(G)\longrightarrow E^r_{0,n}(G),
\qquad
2\leq r\leq n+1.
\]
In other words, the kernel of our edge map
$
E^2_{0,n}(G)\to E^\infty_{0,n}(G)
$
admits a finite filtration whose successive quotients are isomorphic to $\im \partial^r_{r, n-r+1}(G)$ for $2 \leq r \leq n+1$.
It is therefore enough to show that each of these images is finitely
generated as a \(\Z G\)-module.

First let \(2\leq r\leq n\). We claim that
\begin{equation} \label{StabilizationClaim}
E^r_{r,n-r+1}(G)=E^1_{r,n-r+1}(G).
\end{equation}
The incoming differentials to $E^*_{r, n-r+1}(G)$ vanish since their source has total degree \(n+2\), while
\(A^G_{n+2}=0\). 
For \(1\leq t<r\), we thus have
\[
E^{t+1}_{r, n-r+1}(G) = \ker \partial^t_{r, n-r+1} (G) = E^{t}_{r, n-r+1}(G) 
\]
namely $\partial^t_{r, n-r+1} (G) = 0$ because its target is
$
E^t_{r-t,n-r+t}(G),
$
which vanishes for $t \geq 1$ by the exactness of \cref{eq:wall-vertical-augmentation}.
Our claim now follows by induction on $t$.

Since \(A^G_{n+2}=0\) and \(\Z G\) is free as a right \(\Z N\)-module, we get from the exactness of \cref{Nresol} that
\[
E^1_{r,n-r+1}(G)
=
\ker\bigl(
d^0_{r,n-r+1}\colon W^G_{r,n-r+1}\to W^G_{r,n-r}
\bigr) = \Z G^{(\Omega_r)}
\otimes_{\Z N}
\ker
d^N_{n-r+1} = \Z G^{(\Omega_r)}
\otimes_{\Z N} \im d^N_{n-r+2}.
\]
The module \(\Z N^{(X_{n-r+2})}\) is finite-rank free, because
\(n-r+2\leq n\), so $\im d^N_{n-r+2}$ is finitely generated as a
\(\Z N\)-module. 
Moreover, \(\Omega_r\) is finite so in view of our claim, the \(\Z G\)-module
$
E^r_{r,n-r+1}(G)
$
is finitely generated, hence the image of $\partial^r_{r, n-r+1}(G)$ is finitely generated as a $\Z G$-module.

It remains to handle $r = n+1$, where it suffices to show that $E^{n+1}_{n+1,0}(G)$ is finitely generated as a $\Z G$-module.
Since
\(A^G_{n+2}=0\), we have
\[
E^1_{n+1,0}(G)=W^G_{n+1,0}\cong \Z G^{(\Omega_{n+1})}.
\]
Using \cref{eq:wall-vertical-augmentation} for $H = G$, $L = Q$, and $p=n$ we also make the identification
\[
E^1_{n,0}(G) = \coker \bigl(d^0_{n,1} \colon W^G_{n,1} \to W^G_{n,0}\bigr) = \Z Q^{(\Omega_n)}.
\]
With these identifications $\partial^1_{n+1,0}(G) \colon E^1_{n+1,0}(G) \to E^1_{n,0}(G)$ becomes the composition
\[
 \Z G^{(\Omega_{n+1})}
\xrightarrow{\bar\phi^{(\Omega_{n+1})}}
\Z Q^{(\Omega_{n+1})}
\xrightarrow{d^Q_{n+1}}
\Z Q^{(\Omega_n)},
\]
whose kernel is 
$
E^2_{n+1,0}(G).
$
We claim that
$
E^2_{n+1,0}(G)
$
is finitely generated as a $\Z G$-module.


Since $\bar\phi^{(\Omega_{n+1})}$ is onto, we have a short exact sequence 
\[
0\to \ker \bar{\phi}^{(\Omega_{n+1})} \to E^2_{n+1,0}(G) \to \ker d^Q_{n+1} \to 0
\]
Since \cref{ResolFQ} is exact, we have $\ker d^Q_{n+1} = \im d^Q_{n+2}$. Now, 
$\im d^Q_{n+2}$ is finitely generated as a \(\Z G\)-module because
\(\Omega_{n+2}\) is finite. Moreover, $\ker \bar{\phi}^{(\Omega_{n+1})}$ is finitely generated by the exactness of \cref{eq:wall-vertical-augmentation} for $H = G$, $L = Q$, $p=n+1$, the image of the homomorphism of $G$-modules
\[
d^0_{n+1,1}\colon W^G_{n+1,1}\to W^G_{n+1,0}
\]
so it is also finitely generated as a $\Z G$-module, in view of the finite generation of 
$
W^G_{n+1,1} \cong \Z G^{(\Omega_{n+1}\times X_1)}
$
as a $\Z G$-module, stemming from the finiteness of \(\Omega_{n+1}\) and \(X_1\).
The claimed finite generation of $E^2_{n+1,0}(G)$ as a $\Z G$-module follows.

An argument similar to the proof of \cref{StabilizationClaim} shows that
$
E^2_{n+1,0}(G) = E^{n+1}_{n+1,0}(G)
$
so \(E^{n+1}_{n+1,0}(G)\) is finitely generated as a \(\Z G\)-module. 
\end{proof}

\begin{lemma} \label{lem:MGfinite}

The left $\Z G$-module $H_n(A^G_*)$ is finitely generated.

\end{lemma}

\begin{proof}

We have an augmented partial resolution of the trivial $\Z G$-module by finite-rank free $\Z G$-modules
\[
A^G_n \longrightarrow A^G_{n-1} \longrightarrow \cdots
\longrightarrow A^G_0 \longrightarrow \Z \longrightarrow0.
\]
Since $G\in\FP_{n+1}$,
\cref{lem:partial-resolution} with $m=n+1$ tells us that
$
\ker(A^G_n \to A^G_{n-1})
$
is finitely generated as a $\Z G$-module, so its quotient
$H_n(A^G_*)$ is a finitely generated $\Z G$-module as well.
\end{proof}

\section{Free second factor}\label{free-second-factor}

\begin{proposition} \label{prop:free-factor}

Let $n$ be a nonnegative integer, $F$ a free group of finite rank, and let 
\[
1\to N\to G\to Q\to1,\qquad1\to R\to F\to Q\to1,
\]
be short exact sequences of groups, where $N$ is $\FP_{n}$,
$G$ is $\FP_{n+1}$ and $Q$ is $\FP_{n+2}$ and finitely presented. Then
the fibre product $G\times_{Q}F$ is $\FP_{n+1}$.
\end{proposition}

\begin{proof}

Keep the notations of the previous sections. The combination of \cref{KerFinGen} and \cref{lem:MGfinite} implies that $H_n(A^P_*)$ is finitely generated as a $\Z P$-module. By definition we have a short exact sequence of $\Z P$-modules
\[
0 \to \im(A^P_{n+1} \to A^P_n) \to \ker(A^P_n \to A^P_{n-1}) \to H_n(A^P_*) \to  0.
\]
Since $A^P_{n+1}$ is a finite-rank free $\Z P$-module, $\im(A^P_{n+1} \to A^P_n)$ is a finitely generated $\Z P$-module.
It follows that $\ker(A^P_n \to A^P_{n-1})$ is a finitely generated $\Z P$-module as well.
Therefore, 
applying \cref{lem:partial-resolution} to the augmented partial resolution of the trivial $\Z P$-module $\Z$ by finite-rank free $\Z P$-modules
\[
A^P_n\longrightarrow A^P_{n-1}\longrightarrow\cdots
\longrightarrow A^P_0\longrightarrow\Z\longrightarrow0
\]
we get that
$G\times_{Q}F\in\FP_{n+1}$.
\end{proof}

\section{General fibre product}\label{general-fibre}
    
We use the following theorem of Kochloukova and Lima.

\begin{thm}[{\cite[Theorem~3.4]{KocLim}}]\label{thm:KL-transfer}
Let $n \geq 1$ and consider the commutative diagram of groups having exact rows
\[
\begin{tikzcd}
	1 & A & {B_0} & {C_0} & 1 \\
	1 & A & B & C & 1
	\arrow[from=1-1, to=1-2]
	\arrow[from=1-2, to=1-3]
	\arrow[equal, from=1-2, to=2-2]
	\arrow[from=1-3, to=1-4]
	\arrow[from=1-3, to=2-3]
	\arrow[from=1-4, to=1-5]
	\arrow[from=1-4, to=2-4]
	\arrow[from=2-1, to=2-2]
	\arrow[from=2-2, to=2-3]
	\arrow[from=2-3, to=2-4]
	\arrow[from=2-4, to=2-5]
\end{tikzcd}\]
with $A\in\FP_n$, $B_0\in\FP_{n+1}$, and $C\in\FP_{n+1}$. Then
$B\in\FP_{n+1}$.
\end{thm}


\begin{proof}[Proof of \Cref{thm:mainthm}]
The case $n=0$ is the $0$-$1$-$2$ Lemma (\cite[Lemma 3.8]{Kuc12}), so we may assume $n\ge1$. 

Since $G_2\in\FP_{n+1}$, it is finitely generated, so there exists a finite-rank free group
$F$ and a surjective group homomorphism $p \colon F \to G_2$. 
We can then form the commutative diagram with exact rows
\[\begin{tikzcd}
	1 & {N_1} & {G_1 \times_Q F} & F & 1 \\
	1 & {N_1} & {G_1 \times_Q G_2} & {G_2} & 1
	\arrow[from=1-1, to=1-2]
	\arrow[from=1-2, to=1-3]
	\arrow[equal, from=1-2, to=2-2]
	\arrow[from=1-3, to=1-4]
	\arrow["{(\mathrm{id}_{G_1}, p)}", from=1-3, to=2-3]
	\arrow[from=1-4, to=1-5]
	\arrow["p", from=1-4, to=2-4]
	\arrow[from=2-1, to=2-2]
	\arrow[from=2-2, to=2-3]
	\arrow[from=2-3, to=2-4]
	\arrow[from=2-4, to=2-5]
\end{tikzcd}\]
where the group homomorphism from $F$ to $Q$ is $\phi_2 \circ p$. Apply \Cref{prop:free-factor} and \Cref{thm:KL-transfer}.
\end{proof}

\section{Disproving the Homological $n$-$(n+1)$-$(n+2)$ Conjecture}\label{sec:counter}

For a group $G_{1}$ and a subgroup $S\leqslant G_{1}$, we denote
by $\mathrm{HNN}(G_{1},S)$ the HNN extension corresponding to the
identity map $\mathrm{id}_{S}$. By \cite[Proposition 2.13(b)]{Bieri1981},
if $G_{1}$ is $\mathrm{FP}_{n}$ for some nonnegative integer $n$,
then $\mathrm{HNN}(G_{1},S)$ is $\mathrm{FP}_{n}$ if and only if
$S$ is $\mathrm{FP}_{n-1}$. We will repeatedly use the fact 
\[
\mathrm{HNN}(G_{1},S)\cong G_{1}\ltimes\fr{G_{1}/S}
\]
where $\fr{G_{1}/S}$ is the free group on the set $G_{1}/S$, and
the action is by permuting the free basis. 

\begin{proof}[Proof of \Cref{ex:counter}]
Let $Q$ be an $\mathrm{FP}_{\infty}$ group that is not finitely
presented \cite{BestvinaBrady97}, and let $\pi:F\twoheadrightarrow Q$
be a surjection, where $F$ is a finitely generated free group. Set
$A_{0}=F$, $p_{0}=\pi$ and $S_{0}=\graph{p_{0}}=\left\{ (x,\pi(x))\in F\times Q\middle|x\in F\right\} $.
Define recursively 
\[
A_{n+1}=\mathrm{HNN}(A_{n}\times Q,S_{n})
\]
and 
\[
S_{n+1}=\graph{p_{n+1}},
\]
where $p_{n+1}:A_{n+1}\twoheadrightarrow Q$ is defined on $A_{n}\times Q$
by the projection onto $Q$, and by sending the stable letter to $1_{Q}$.
Observe $p_n$ is surjective for every $n$, that $S_{n}\cong A_{n}$ and that both $Q$ and $F$ are $\mathrm{FP}_{\infty}$.
It is therefore easy to see by induction (using \cite[Proposition 2.13(b)]{Bieri1981} and the fact property $\mathrm{FP}_\infty$ is preserved under extensions, and in particular under direct products)
that $A_{n}$ is $\mathrm{FP}_{\infty}$ for every $n$. Under the identification
\[
A_{n+1}\cong(A_{n}\times Q)\ltimes\fr{\left(A_{n}\times Q\right)/S_{n}},
\]
the map $p_{n+1}$ is given by first quotienting out $\fr{A_{n}\times Q/S_{n}}$
and then projecting onto the $Q$ coordinate. We therefore get 
\[
\ker p_{n+1}=(A_{n}\times\left\{ 1\right\} )\ltimes\fr{\left(A_{n}\times Q\right)/S_{n}}.
\]
The action $A_{n}\times\left\{ 1\right\} \curvearrowright(A_{n}\times Q)/S_{n}$
is transitive with stabiliser $\ker p_{n}$, so 
\[
\ker p_{n+1}\cong A_{n}\ltimes\fr{A_{n}/\ker p_{n}}\cong\mathrm{HNN}(A_{n},\ker p_{n}).
\]
Since $A_{n}$ is $\mathrm{FP}_{\infty}$, one proves inductively
(by \cite[Proposition 2.13(b)]{Bieri1981}) that $\ker p_{n}$ is
$\mathrm{FP}_{n}$ for every $n$. We now need to show (for every
$n$) that $A_{n}\times_{Q}F$ is not $\mathrm{FP}_{n+1}$ . 

Since the map $p_{n+1}:A_{n+1}\twoheadrightarrow Q$ is trivial on the normal subgroup $\fr{\left(A_{n}\times Q\right)/S_{n}}$, we have
\[
A_{n+1}\times_{Q}F\cong\left((A_{n}\times Q)\times_{Q}F\right)\ltimes\fr{\left(A_{n}\times Q\right)/S_{n}}.
\]
We moreover have $(A_{n}\times Q)\times_{Q}F\cong A_{n}\times F$,
so 
\[
A_{n+1}\times_{Q}F\cong\left(A_{n}\times F\right)\ltimes\fr{\left(A_{n}\times Q\right)/S_{n}},
\]
where the action of $A_{n}\times F$ on $\left(A_{n}\times Q\right)/S_{n}$
is given by
\[
(x,y).(a,b)S_n=(xa,\pi(y)b)S_{n}.
\]
This is a transitive action with stabiliser 
\begin{align*}
\left\{ (x,y)\in A_{n}\times F\middle|(x,\pi(y))\in S_{n}\right\}  & =\left\{ (x,y)\in A_{n}\times F\middle|p_{n}(x)=\pi(y)\right\} ,
\end{align*}
i.e., the stabiliser is exactly $A_{n}\times_{Q}F$. We therefore
get
\[
A_{n+1}\times_{Q}F\cong\mathrm{HNN}(A_{n}\times F,A_{n}\times_{Q}F).
\]
We now prove by induction that $A_{n}\times_{Q}F$ is not $\mathrm{FP}_{n+1}$.
For $n=0$, we have $A_{0}\times_{Q}F=F\times_{Q}F$, which is isomorphic
to $F\ltimes\ker\pi$ (via $(x,y)\mapsto(x,x^{-1}y)$). It is not
finitely generated since, if $(a_{1},b_{1}),\dots,(a_{r},b_{r})$
generated $F\ltimes\ker\pi$, we would get that $b_{1},\dots,b_{r}$
normally generate $\ker\pi$, and hence that $Q=F/\ker\pi$ is finitely
presented, a contradiction. Now, assume $A_{n}\times_{Q}F$ is not
$\mathrm{FP}_{n+1}$ for some nonnegative integer $n$. Since $A_{n}\times F$
is $\mathrm{FP}_{\infty}$ and $A_{n}\times_{Q}F$ is not $\mathrm{FP}_{n+1}$,
we get (by \cite[Proposition 2.13(b)]{Bieri1981}) that $A_{n+1}\times_{Q}F\cong\mathrm{HNN}(A_{n}\times F,A_{n}\times_{Q}F)$
is not $\mathrm{FP}_{n+2}$, as needed.
\end{proof}

\printbibliography
\vspace{0.5cm}

\noindent{\textsc{Department of Mathematics, Weizmann Institute of Science, 234 Herzl Street, Rehovot 7610001, Israel}}

\vspace{0.5cm}

\noindent{\textit{Email address:} \texttt{tal.cohen@weizmann.ac.il}}

\noindent{\textit{Email address:} \texttt{mark.shusterman@weizmann.ac.il}}
\end{document}